\documentclass[12pt,a4paper]{article}

\usepackage{mathrsfs}
\usepackage{amscd}
\usepackage{hyperref}
\hypersetup{colorlinks=true,citecolor=green,linkcolor=blue,urlcolor=blue}
\usepackage{tipa}
\usepackage{amssymb}
\usepackage{amsfonts}
\usepackage{stmaryrd}
\usepackage{amssymb}

\usepackage{CJK}  
\usepackage{indentfirst} 

\usepackage{latexsym}   
\usepackage{bm}         

\usepackage{pifont}
\usepackage{bm}
\usepackage{amsmath,amssymb,amsfonts}
\usepackage{indentfirst}

\usepackage{xcolor}
\usepackage{cases}
\usepackage{wasysym}
\usepackage{amsthm}

\newtheorem{theorem}{Theorem}[section]
\newtheorem{lemma}[theorem]{Lemma}
\newtheorem{proposition}[theorem]{Proposition}

\newtheorem{remark}[theorem]{Remark}
\newtheorem{corollary}[theorem]{Corollary}

\CJKtilde   

\begin{document}

\title{\Large \textbf{Freeness of spherical Hecke modules of unramified $U(2, 1)$ in characteristic $p$}}
\date{}
\author{\textbf{Peng Xu}}
\maketitle

\begin{abstract}
Let $F$ be a non-archimedean local field of odd residue characteristic $p$. Let $G$ be the unramified unitary group $U(2, 1)(E/F)$ in three variables, and $K$ be a maximal compact open subgroup of $G$. For an irreducible smooth representation $\sigma$ of $K$ over $\overline{\mathbf{F}}_p$, we prove that the compactly induced representation $\text{ind}^G _K \sigma$ is free of infinite rank over the spherical Hecke algebra $\mathcal{H}(K, \sigma)$.
\end{abstract}

\section{Introduction}
\label{section: intro}

In the last two decades, the area of $p$-modular representations of $p$-adic reductive groups has already had a vast development. In their pioneering work (\cite{B-L95}, \cite{B-L94}), Barthel--Livn$\acute{\textnormal{e}}$ gave a classification of irreducible smooth $\overline{\mathbf{F}}_p$-representations of the group $GL_2$ over a local field of residue characteristic $p$, leaving the supersingular representations as a mystery. Almost ten years later, C. Breuil (\cite{Breuil03}) classified the supersingular representations of $GL_2 (\mathbf{Q}_p)$, which was one of the starting points of the mod-$p$ and $p$-adic local Langlands program, initiated and developed by Breuil and many other mathematicians (see \cite{Breuil10} for an overview). Recently, a classification of irreducible \emph{admissible} mod-$p$ representations of $p$-adic reductive groups has been obtained by Abe--Henniart--Herzig--Vign$\acute{\textnormal{e}}$ras (\cite{AHHV17}), as a generalization of many previous works due to these authors. However, our knowledge of the mysterious supersingular representations is still very limited, and from the work of Breuil and Pa$\check{\textnormal{s}}$k$\bar{\textnormal{u}}$nas (\cite{BP12}) it seems that a \lq classification' of that, even for the group $GL_2 (F)$ with $F\neq \mathbf{Q}_p$, is out of reach.

Smooth representations induced from maximal compact open subgroups, and their associated spherical Hecke algebras, are central objects in the study of $p$-modular representation theory of $p$-adic reductive groups, see \cite{B-L95}, \cite{Her2011b}, \cite{Abe-2011}, \cite{AHHV17}.  As such an induced representation is naturally a left module over its spherical Hecke algebra, a general question is to ask what we can say about the nature of this module ?  Certainly, we might not expect much without any restrictions on the group under consideration. The only general result in this direction, as far as we know, is due to Gro${\ss}$e-Kl\"{o}nne (\cite{GK14}), see Remark \ref{GK's universal modules} for a precise description of his result. In this note, we investigate this question for the unitary group in three variables.

We start to describe our result in detail. Let $E/F$ be an unramified quadratic extension of non-archimedean local fields with odd residue characteristic $p$. Let $G$ be the unitary group $U(2,1)(E/F)$ defined over $F$, and $K$ be a maximal compact open subgroup of $G$. For an irreducible smooth representation $\sigma$ of $K$ over $\overline{\mathbf{F}}_p$, the compactly induced representation $\textnormal{ind}^G _K \sigma$ is naturally a left module over the spherical Hecke algebra $\mathcal{H}(K, \sigma) :=\textnormal{End}_G (\textnormal{ind}^G _K \sigma)$. Our main result is as follows:

\begin{theorem}\label{main}(Theorem \ref{main theorem})

The compactly induced representation $\textnormal{ind}^G _K \sigma$ is a free module of infinite rank over $\mathcal{H}(K, \sigma)$.

\end{theorem}

Theorem \ref{main} is an analogue to a theorem of Barthel--Livn$\acute{\textnormal{e}}$ (\cite[Theorem 19]{B-L95}) for the group $GL_2 (F)$, and we follow their approach in general. The underlying idea is na\"{i}ve and depends on an analysis of the Bruhat--Tits tree of the group $G$. However, there is an essential difference when we prove a key ingredient (Lemma \ref{key lemma for freeness}) in our case, where we find a more conceptual approach and reduce it to some simple computations on the tree of $G$.

Our freeness result has some natural applications, and we record some of them in this note.

The first one is that every non-trivial spherical universal module of $G$ is infinite dimensional (Corollary \ref{inf dim}), which at least implies the existence of supersingular representations of $G$ containing a given irreducible smooth representation of $K$. The existence of supersingular representations, to our knowledge, was only proved very recently for most simple adjoint $p$-adic group (\cite{V2017}).

Next, following Gro${\ss}$e-Kl\"{o}nne (\cite[section 9]{GK14}), we apply Theorem \ref{main} to investigate integral structures in certain $p$-adic locally algebraic representations of $G$, and we formulate a conditional result for irreducible tamely ramified principal series (Theorem \ref{invariant norm}).

This note is organized as follows. In section \ref{section: notation}, we set up the general notations and review some necessary background on the group $G$ and its Bruhat--Tits tree. In section \ref{section: spherical hecke}, we study the Hecke operator $T$ in detail, and describe the image of certain invariant subspace of $\textnormal{ind}^G _K \sigma$ under $T$. In section \ref{section: proof of main}, we prove our main result. In section \ref{section: invariant norm}, we apply our main result to investigate $G$-invariant norms in certain local algebraic representations of $G$. In the final appendix \ref{section: proof of 3.4}, we provide a detail proof of the recursion relations in the spherical Hecke algebra of $G$.

\section{Notations and Preliminaries}\label{section: notation}

Let $F$ be a non-archimedean local field of odd residue characteristic $p$, with ring of integers $\mathfrak{o}_F$ and maximal ideal $\mathfrak{p}_F$, and let $k_F$ be its residue field of cardinality $q=p^f$. Fix a separable closure $F_s$ of $F$. Let $E$ be the unramified quadratic extension of $F$ in $F_s$. We use similar notations $\mathfrak{o}_E$, $\mathfrak{p}_E$, $k_E$ for analogous objects of $E$. Let $\varpi_{E}$ be a uniformizer
of $E$, lying in $F$. Given a 3-dimensional vector space $V$ over $E$, we identify it with $E^{3}$ (the usual column space in three variables), by fixing a basis of $V$.
Equip $V$ with the non-degenerate Hermitian form h:
\begin{center}
 $\text{h}: V \times V \rightarrow E$,
$(v_1, v_2) \mapsto v_1 ^\textnormal{T}\beta \overline{v_2}, v_1, v_2\in V$.
\end{center}
Here, $-$ denotes the non-trivial Galois conjugation on $E/F$, inherited by $V$, and
$\beta$ is the matrix
\begin{center}
$\begin{pmatrix} 0  & 0 & 1\\
0  & 1 & 0\\
1 & 0 & 0
\end{pmatrix}$
\end{center}

The unitary group $G$ is the subgroup of $GL(3, E)$ whose elements fix the Hermitian form h:

\begin{center}
$G=\{g\in \textnormal{GL}(3, E)\mid \textnormal{h}(gv_{1}, gv_{2})= \textnormal{h}(v_{1}, v_{2}), \textnormal{for~any}~v_{1}, v_{2}~\in V\}.$
\end{center}

Let $B=HN$ (resp, $B'= HN'$) be the subgroup of upper (resp, lower) triangular matrices of $G$, where $N$ (resp, $N'$) is the unipotent radical of $B$ (resp, $B'$) and $H$ is the diagonal subgroup of $G$. Denote an element of the following form in $N$ and $N'$ by $n(x, y)$ and $n'(x, y)$ respectively:
\begin{center}
$\begin{pmatrix}  1 & x & y  \\ 0 & 1 & -\bar{x}\\
0 & 0 & 1
\end{pmatrix}$, ~
$\begin{pmatrix}  1 & 0 & 0   \\ x & 1 & 0\\
y & -\bar{x} & 1
\end{pmatrix}$,
\end{center}
where $(x, y)\in E^2$ satisfies $x\bar{x}+ y+ \bar{y}=0$. Denote by $N_k$ (resp, $N'_k$), for any $k\in \mathbb{Z}$, the subgroup of $N$ (resp, $N'$) consisting of $n(x, y)$ (resp, $n'(x, y)$) with $y\in \mathfrak{p}^{k}_E$. For $x\in E^\times$, denote by $h(x)$ an element in $H$ of the following form:
\begin{center}
$\begin{pmatrix}  x & 0 & 0  \\ 0 & -\bar{x}x^{-1} & 0\\
0 & 0 & \bar{x}^{-1}
\end{pmatrix}.$
\end{center}

We record the following useful identity in $G$: for $y\neq 0$,
\begin{equation}\label{useful identity}
\beta n(x, y)= n(\bar{y}^{-1}x, y^{-1})\cdot h(\bar{y}^{-1})\cdot n'(-\bar{y}^{-1}\bar{x}, y^{-1}).
\end{equation}

\medskip
Up to conjugacy, the group $G$ has two maximal compact open subgroups $K_0$ and $K_1$ (\cite{Hij63},\cite{Tits79}), which are given by:
\begin{center}
$K_0= \begin{pmatrix}  \mathfrak{o}_E & \mathfrak{o}_E & \mathfrak{o}_E  \\ \mathfrak{o}_E  & \mathfrak{o}_E & \mathfrak{o}_E\\
\mathfrak{o}_E & \mathfrak{o}_E & \mathfrak{o}_E
\end{pmatrix}\cap G, ~K_1= \begin{pmatrix}  \mathfrak{o}_E & \mathfrak{o}_E & \mathfrak{p}^{-1}_E  \\ \mathfrak{p}_E  & \mathfrak{o}_E & \mathfrak{o}_E\\
\mathfrak{p}_E & \mathfrak{p}_E & \mathfrak{o}_E
\end{pmatrix}\cap G.$
\end{center}
The maximal normal pro-$p$ subgroups of $K_0$ and $K_1$ are respectively:

$K^1 _0= 1+\varpi_E M_3 (\mathfrak{o}_E)\cap G, ~K^1 _1= \begin{pmatrix}  1+\mathfrak{p}_E & \mathfrak{o}_E & \mathfrak{o}_E  \\ \mathfrak{p}_E  & 1+\mathfrak{p}_E & \mathfrak{o}_E\\
\mathfrak{p}^2_E & \mathfrak{p}_E & 1+\mathfrak{p}_E
\end{pmatrix}\cap G.$

Let $\alpha$ be the following diagonal matrix in $G$:
\[ \begin{matrix}\begin{pmatrix} \varpi_{E}^{-1}  & 0 & 0  \\ 0  & 1 & 0\\
0 & 0 & \varpi_{E}
\end{pmatrix}
\end{matrix} ,\]
and put $\beta'=\beta \alpha^{-1}$. Note that $\beta\in K_0$ and $\beta'\in K_1$.

 Let $K$ be one of the two maximal compact open subgroups of $G$ above, and $K^1$ be the maximal normal pro-$p$ subgroup of $K$. We identify the finite group $\Gamma_K= K/K^1$ with the $k_F$-points of an algebraic group defined over $k_F$, denoted also by $\Gamma_K$: when $K$ is $K_0$, $\Gamma_K$ is $U(2, 1)(k_E /k_F)$, and when $K$ is $K_1$, $\Gamma_K$ is $U(1, 1)\times U(1)(k_E /k_F)$. Let $\mathbb{B}$ (resp, $\mathbb{B}'$) be the upper (resp, lower) triangular subgroup of $\Gamma_K$, and $\mathbb{U}$ (resp, $\mathbb{U}'$) be its unipotent radical. The Iwahori subgroup $I_K$ (resp, $I'_K$) and pro-$p$ Iwahori subgroup $I_{1, K}$ (resp, $I'_{1, K}$) in $K$ are the preimages of $\mathbb{B}$ (resp, $\mathbb{B}'$) and $\mathbb{U}$ (resp, $\mathbb{U}'$) in $K$. We have the following Bruhat decomposition for $K$:
 \begin{center}
 $K= I \cup I \beta_K I$,
 \end{center}
where $\beta_K$ denotes the unique element in $K\cap \{\beta, \beta'\}$, $I$ is either $I_K$ or $I'_K$.
\medskip

We end this part by recalling some facts on the Bruhat--Tits tree $\triangle$ of $G$. Denote by $X_0$ the set of vertices of $\triangle$, which consists of all $\mathfrak{o}_E$-lattices $\mathcal{L}$ in $E^3$, such that

\begin{center}
$\varpi_E \mathcal{L} \subseteq\mathcal{L}^\ast \subseteq \mathcal{L}$,
\end{center}
where $\mathcal{L}^\ast$ is the dual lattice of $\mathcal{L}$ under the Hermitian form $h$, i.e., $\mathcal{L}^\ast= \{v\in E^3\mid h(v, \mathcal{L})\in \mathfrak{p}_E\}$.

 Let $\mathbf{v}, ~\mathbf{v}'$ be two vertices in $X_0$ represented by $\mathcal{L}$ and $\mathcal{L}'$.  The vertices $\mathbf{v}$ and $\mathbf{v}'$ are \emph{adjacent}, if:
 \begin{center}
 $ \mathcal{L}'\subset \mathcal{L}$ or $\mathcal{L}\subset \mathcal{L}'$.
 \end{center}
When $\mathbf{v}$ and $\mathbf{v}'$ are adjacent, we have the edge $(\mathbf{v}, \mathbf{v}')$ on the tree.

Let $\{e_{-1}, e_{0}, e_{1}\}$ be the standard basis of $E^{3}$. We consider the following two lattices in $E^{3}$:

\begin{center}
$\mathcal {L}_{0}=\mathfrak{o}_{E}e_{-1}\oplus\mathfrak{o}_{E}e_{0}\oplus\mathfrak{o}_{E}e_{1}$, ~$\mathcal
{L}_{1}=\mathfrak{o}_{E}e_{-1}\oplus\mathfrak{o}_{E}e_{0}\oplus\mathfrak{p}_{E}e_{1}.$
\end{center}
Denote respectively by $\mathbf{v}_{0}, \mathbf{v}_{1}$ the vertices represented by $\mathcal {L}_{0}$ and $\mathcal
{L}_{1}$, which are then adjacent. The group $G$ acts on $X_0$ in a natural way with two orbits, i.e.,
\begin{center}
$X_0= \{G \cdot\mathbf{v}_0\}\cup \{G \cdot\mathbf{v}_1\}$.
\end{center}
For $i=0, 1$, the stabilizer of $\mathbf{v}_i$ in $G$ is exactly the maximal open compact subgroup $K_i$, and the stabilizer of the edge $(\mathbf{v}_0, \mathbf{v}_1)$ is the intersection $K_0 \cap K_1$.

 For a vertex $\mathbf{v}\in X_0$, the number of vertices adjacent to $\mathbf{v}$ is equal to $q^{c_\mathbf{v}} +1$, where $c_\mathbf{v}$ is either $3$ or $1$, depending on $\mathbf{v}\in \{G \cdot\mathbf{v}_0\}$ or $\{G \cdot\mathbf{v}_1\}$. For a maximal compact open subgroup $K$, we will write $c_K$ for $c_{\mathbf{v}}$, if $\mathbf{v}$ is the unique vertex on the tree stabilized by $K$.

\medskip
Unless otherwise stated, all the representations of $G$ and its subgroups considered in this note are smooth over $\overline{\mathbf{F}}_p$.

\section{The spherical Hecke operator $T$}\label{section: spherical hecke}

\subsection{The spherical Hecke algebra $\mathcal{H}(K, \sigma)$}

 Let $K$ be a maximal compact open subgroup of $G$, and $(\sigma, W)$ be an irreducible smooth representation of $K$. As $K^1$ is pro-$p$, $\sigma$ factors through the finite group $\Gamma_K= K/K^1$, i.e., $\sigma$ is the inflation of an irreducible representation of $\Gamma_K$.

  It is well-known that $\sigma^{I_{1,K}}$ and $\sigma_{I'_{1,K}}$ are both one-dimensional, and that the natural composition map $\sigma^{I_{1,K}}\hookrightarrow \sigma \twoheadrightarrow \sigma_{I'_{1,K}}$ is non-zero, i.e., an isomorphism of vector spaces (\cite[Theorem 6.12]{C-E2004}). Denote by $j_\sigma$ the inverse of the composition map just mentioned. For $v\in \sigma^{I_{1,K}}$, we have $j_\sigma (\bar{v})= v$, where $\bar{v}$ is the image of $v$ in $\sigma_{I'_{1,K}}$. When viewed as a map in $\text{Hom}_{\overline{\mathbf{F}}_p}(\sigma, \sigma^{I_{1, K}})$, the $j_\sigma$ factors through $\sigma_{I'_{1,K}}$, i.e., it vanishes on $\sigma (I'_{1,K})$.

\begin{remark}\label{value of lambda}
There is a unique constant $\lambda_{\beta_K, \sigma}\in \overline{\mathbf{F}}_p$, such that $\beta_K \cdot v-\lambda_{\beta_K, \sigma}v\in \sigma (I'_{1,K})$, for $v\in \sigma^{I_{1,K}}$. The value of $\lambda_{\beta_K, \sigma}$ is known: it is zero unless $\sigma$ is a character (\cite[Proposition 3.16]{H-V2011}), due to the fact that $\beta_K\notin I_K\cdot I'_K$. When $\sigma$ is a character, $\lambda_{\beta_K, \sigma}$ is just the scalar $\sigma (\beta_K)$.
\end{remark}

\begin{remark}\label{n_K and m_K}
There are unique integers $n_K$ and $m_K$ such that $N\cap I_{1, K}= N_{n_K}$ and $N'\cap I_{1, K}=N'_{m_K}$.
\end{remark}

Let $\text{ind}_K ^{G}\sigma$ be the compactly induced smooth representation, i.e., the representation of $G$ with underlying space $S(G, \sigma)$
\begin{center}
$S(G, \sigma)=\{f: G\rightarrow W\mid  f(kg)=\sigma (k)\cdot f(g),~\text{for~any}~k\in K~\text{and}~g\in G,~ \text{locally~constant~with~compact~support}\}$
\end{center}
and $G$ acting by right translation. In this note, we will sometimes call $\text{ind}_K ^{G}\sigma$ a \emph{maximal compact induction}.

As usual (\cite[section 2.3]{B-L95}), denote by $[g, v]$ the function in $S(G, \sigma)$, supported on $Kg^{-1}$ and having value $v\in W$ at $g^{-1}$.  An element $g'\in G$ acts on the function $[g, v]$ by $g'\cdot[g, v]= [g'g, v]$, and we have $[gk, v]= [g, \sigma(k)v]$ for $k\in K$.

The spherical Hecke algebra $\mathcal{H}(K, \sigma)$ is defined as $\text{End}_G (\text{ind}^G _K \sigma)$, and by \cite[Proposition 5]{B-L95} it is isomorphic to the convolution algebra $\mathcal{H}_K (\sigma)$ of all compactly support and locally constant functions $\varphi$ from $G$ to $\text{End}_{\overline{\mathbf{F}}_p}(\sigma)$, satisfying $\varphi(kgk')=\sigma(k)\varphi(g)\sigma(k')$ for any $g\in G$ and $k, k'\in K$. Let $\varphi$ be the function in $\mathcal{H}_K (\sigma)$, supported on $K\alpha K$, and satisfying $\varphi (\alpha)= j_\sigma$. Denote by $T$ the Hecke operator in $\mathcal{H}(K, \sigma)$, which corresponds to the function $\varphi$, via the isomorphism between $\mathcal{H}_K (\sigma)$ and  $\mathcal{H}(K, \sigma)$.

When $K$ is hyperspecial, the following proposition is a special case of a theorem of Herzig (\cite{Her2011a}).

\begin{proposition}
The algebra $\mathcal{H}(K, \sigma)$ is isomorphic to $\overline{\mathbf{F}}_p [T]$.
\end{proposition}

\begin{proof}
Here, we give a straightforward proof by explicit computations, and the recursion relations in the algebra will be used later.

It suffices to consider the algebra $\mathcal{H}_K (\sigma)$.
Recall the Cartan decomposition of $G$:
\begin{center}
$G=\bigcup_{n\geq 0}~K \alpha^n K$.
\end{center}

Let $\varphi$ be a function in $\mathcal{H}_K (\sigma)$, supported on the double coset $K \alpha^n K$. Then, for any $k_1, ~k_2\in K$, satisfying $k_1 \alpha^n= \alpha^n k_2$, we are given $\sigma(k_1)\varphi(\alpha^n)=\varphi(\alpha^n)\sigma(k_2)$. When $n=0$, $\varphi(Id)$ commutes with all $\sigma(k)$. As $\sigma$ is irreducible, by Schur's Lemma $\varphi(Id)$ is a scalar.

For $n>0$, let $k_1\in N'_{2n-1+ m_K}$. As $k_1\in K^1$, $\sigma(k_1)=1$. Now $k_2=\alpha^{-n}k_1 \alpha^n \in N'_{m_K -1}$, and we have $\varphi(\alpha^n)=\varphi(\alpha^n)\cdot \sigma(k_2)$. We see $\varphi(\alpha^n)$ factors through $\sigma_{I'_{1, K}}$. Similarly, for $k_1\in N_{n_K}$, $k_2= \alpha^{-n}k_1 \alpha^n\in K^1$, and we get $\sigma(k_1)\varphi(\alpha^n)=\varphi(\alpha^n)$,  that is to say $\text{Im}(\varphi(\alpha^n))\in \sigma^{I_{1, K}}$. In other words, $\varphi(\alpha^n)$ only differs from $j_{\sigma}$ by a scalar.

For $n\geq 0$, let $\varphi_n$ be the function in $\mathcal{H}_K (\sigma)$, supported on $K \alpha^n K$, determined by its value on $\alpha^n$: $\varphi_0(Id)=Id, ~\varphi_{n}(\alpha^n)= j_{\sigma}, ~n> 0$.

\begin{proposition}\label{basis for h_k_0}
$\{\varphi_n\}_{n\geq 0}$ consists of a basis of  $\mathcal{H}_K (\sigma)$, and they satisfy the following convolution relations: for $n\geq 1,~ l\geq 0$,
\begin{equation}
\varphi_1 \ast \varphi_{n}(\alpha^l)=\begin{cases}
0, ~~~~~~~~~l\neq n,~n+1;\\
c\cdot j_{\sigma}, ~~~l=n;\\
j_{\sigma}, ~~~~~~~l=n+1,
\end{cases}
\end{equation}
where $c$ is some constant depending on $\sigma$.
\end{proposition}
\begin{proof}
 The convolution formulae in the proposition give that $\varphi_1 \ast \varphi_{n}= c\cdot \varphi_n+ \varphi_{n+1}$, which will matter to us later. In particular, it follows that the algebra $\mathcal{H}_K (\sigma)$ is commutative. We leave the proof to the appendix \ref{section: proof of 3.4}.
\end{proof}

Denote by $T_n$ the operator in $\mathcal{H}(K, \sigma)$ which corresponds to $\varphi_n$. We then have similar composition of relations among $\{T_n\}_{n\geq 0}$, namely
\begin{center}
$T_1 \cdot T_n = c\cdot T_n + T_{n+1}$,
\end{center}
and the assertion in the proposition follows.
\end{proof}

\subsection{The formula $T[Id, v_0]$}

Let $v$ be a vector in $V$, and by \cite[(8)]{B-L95} we have
\begin{equation}\label{general formular for T}
T~[Id, v]=\sum_{g\in G/K} [g, \varphi(g^{-1})\cdot v].
\end{equation}
As $\varphi$ is supported on the double coset $K\alpha K=K\alpha^{-1} K$, we decompose $K\alpha^{-1} K$ into right cosets of $K$:
\begin{center}
$K\alpha^{-1} K= \bigcup_{k\in K/ (K\cap \alpha^{-1}K\alpha)} k\alpha^{-1}K$,
\end{center}
and we need to identify $K/ (K\cap \alpha^{-1}K\alpha)$ with some simpler set. Note firstly that $I'_K$ contains $K\cap \alpha^{-1}K\alpha$, and that
\begin{center}
$I'_K / (K\cap \alpha^{-1}K\alpha) \cong N_{n_K +1}/N_{n_K +2}$,
\end{center}
where, as mentioned in Remark \ref{n_K and m_K}, $n_K$ is the unique integer such that $N\cap I_{K}= N_{n_K}$.

Secondly, we note the coset decomposition of $K$ with respect to $I'_K$:
\begin{center}
$K= I'_K \cup \bigcup_{u\in N_{n_K}/N_{n_K +1}}\beta_K u I'_K$.
\end{center}
In all, we may identify $K/ (K\cap \alpha^{-1}K\alpha)$ with the following set
\begin{center}
$(N_{n_K +1}/N_{n_K +2})\bigcup \beta_K\cdot (N_{n_K}/N_{n_K +2})$,
\end{center}
and hence with the following set:
\begin{center}
$\beta_K\cdot (N_{n_K +1}/N_{n_K +2})\bigcup (N_{n_K}/N_{n_K +2})$.
\end{center}

 Using the previous identification, the above equation \eqref{general formular for T} becomes:

\begin{equation}\label{explicit form for T}
T~[Id, v]=\sum_{u\in N_{n_K +1}/N_{n_K+ 2}}~[\beta_K u\alpha^{-1}, j_{\sigma}\sigma(\beta_K) v]+ \sum_{u\in N_{n_K}/N_{n_K+ 2}} [u \alpha^{-1}, j_{\sigma}\sigma (u^{-1})v],
\end{equation}
where we note that $N_{n_K +1} \subset K^1$.

\medskip
Recall we have a Cartan--Iwahori decomposition:
\begin{equation}\label{C-I decom}
G= \bigcup_{n\in \mathbb{Z}}~K \alpha^n I_{1, K}.
\end{equation}

Based on \eqref{C-I decom}, we may describe the $I_{1, K}$-invariant subspace of $\text{ind}_{K}^G \sigma$. By Frobenius reciprocity and an argument like that of \cite[Proposition 5]{B-L95}, we have $(\text{ind}_{K}^G \sigma)^{I_{1, K}}=\{f\in S(G, \sigma) \mid  f(kgi)=\sigma(k)f(g),~\text{for}~k\in K, g\in G, i\in I_{1, K}\}$. Let $f$ be a function in $(\text{ind}_{K}^G \sigma)^{I_{1, K}}$, supported in $K \alpha^n I_{1, K}$, so $f$ is determined by its value at $\alpha^n$. For any $k \in K, ~i\in I_{1, K}$ such that $k\alpha^n =\alpha^n i$, $f(\alpha^n)$ should satisfy $\sigma(k)f(\alpha^n)=f(\alpha^n)$.

For $n\geq 0$, and $u\in N_{n_K}$, we get $\sigma(u)f(\alpha^n)=f(\alpha^n)$, that is to say $f(\alpha^n)$ is fixed by $N_{n_K}$. Similarly, for a negative $n$, and $u'\in N\cap I'_{1, K}= N'_{m_K -1}$, we have $\sigma(u')f(\alpha^n)=f(\alpha^n)$, which implies that $f(\alpha^n)$ is fixed by $N\cap I'_{1, K}$. Note that $\sigma^{I_{1, K}}= \sigma^{N_{n_K}}$ and $\sigma^{I'_{1, K}}= \sigma^{N'_{m_K -1}}$, as $I_{1, K}=N_{n_K}\cdot K^1$ and $I'_{1, K}=N'_{m_K -1}\cdot K^1$, where $K^1$ acts trivially on $\sigma$.

\medskip
Recall that the subspace of $I_{1, K}$-invariants in $\sigma$ is one-dimensional, and from now on \emph{we fix a non-zero $v_0\in \sigma^{I_{1, K}}$ throughout this note.}
\medskip

Let $f_n$ be the function in $(\textnormal{ind}_K ^G \sigma)^{I_{1, K}}$, supported on $K \alpha^{-n} I_{1, K}$, such that
\begin{center}
$f_n (\alpha^{-n})=  \begin{cases}
\beta_K\cdot v_0, ~~~~~~~n>0,\\
v_0 ~~~~~~~~~~~~~~~~~n\leq 0.
\end{cases}$
\end{center}
By the remarks above, we have:
\begin{lemma}\label{basis of I_1 of compact induction}
~~The set of functions $\{f_n\mid ~n\in \mathbb{Z}\}$ consists of a basis of the $I_{1, K}$-invariants of the maximal compact induction $\textnormal{ind}^{G} _K\sigma$.
\end{lemma}

It is useful to rewrite the function $f_n$ in terms of a canonical $G$-transition of $f_0= [Id, v_0]$.
\begin{center}
$f_{-m}= \sum_{u\in N_{n_K}/N_{n_K+ 2m}}~[u\alpha^{-m}, ~v_0]$, for $m\geq 0$;

~$f_n= \sum_{u'\in N'_{m_K}/N'_{m_K+2n-1}} [u'\alpha^n, ~\beta_K\cdot v_0]$, for $n> 0$.
\end{center}

\medskip
We now record the following formula $T ([Id, v_0])$, i.e., $T\cdot f_0$, and we will do such thing for all the other $f_n$ in the next subsection.

\begin{proposition}\label{hecke operator formula}

$T\cdot f_0= f_{-1} + \lambda_{\beta_K, \sigma}\cdot f_1$.

\end{proposition}

\begin{proof}

In \eqref{explicit form for T}, we take $v$ as $v_0$. Then the first sum in \eqref{explicit form for T} becomes $\lambda_{\beta_K, \sigma}f_1$, and the second sum is just the function $f_{-1}$, as the group $N_{n_K}$ fixes $v_0$. We are done.
\end{proof}

\subsection{The formula $Tf_n$ for $n\neq 0$}
The purpose of this part is to push Proposition \ref{hecke operator formula} further.

For $n\geq 0$, denote by $R^+ _n (\sigma)$ (resp, $R^{-} _n (\sigma)$) the subspace of functions in $\text{ind}^G _K \sigma$ which are supported in the coset $K\alpha^n I_K$ (resp, $K\alpha^{-(n+1)}I_K$). Both spaces are $I_K$-stable. In our former notations, we indeed have $K\alpha^n I_K= K\alpha^n N_{n_K}$,  and $K\alpha^{-(n+1)}I_K= K\alpha^{-(n+1)}N'_{m_K}$, for $n\geq 0$. Then we may rewrite $R^+ _n (\sigma)$ and $R^{-} _n (\sigma)$ as follows:

\begin{lemma}

For $n\geq 0$, $R^+ _n (\sigma)= [N_{n_K} \alpha^{-n}, \sigma];$

~~~~~~~~~~~~~For $n\geq 1$, $R^- _{n-1} (\sigma) =[N'_{m_K}\alpha^n,  \sigma]$.

\end{lemma}

\begin{remark}\label{basis of R-spaces}
Note that $f_{-n}\in R^+ _n (\sigma)^{I_{1, K}}$ for $n\geq 0$, and $f_n \in R^- _{n-1} (\sigma)^{I_{1, K}}$ for $n\geq 1$. From Lemma \ref{basis of I_1 of compact induction} and its argument, both $R^+ _n (\sigma)^{I_{1, K}}$ and $R^- _{n-1} (\sigma)^{I_{1, K}}$ are one dimensional, hence they are generated by $f_{-n}$ and $f_n$ respectively.
\end{remark}

\medskip
Naturally we are interested in how the above $I_K$-subspaces are changed under the Hecke operator $T$, and the following is the first observation.
\begin{proposition}\label{image of the tree under T}

$(1)$.~~$T (R^+ _0 (\sigma)) \subseteq R^{+}_1 (\sigma)\oplus R^{-}_0 (\sigma)$.

$(2)$.~~$T (R^{+}_n (\sigma)) \subseteq R^{+}_{n-1} (\sigma)\oplus R^{+}_n (\sigma)\oplus R^{+}_{n+1} (\sigma), n \geq 1$.

$(3)$.~~$T (R^{-}_n (\sigma))\subseteq R^{-}_{n-1} (\sigma)\oplus R^{-}_{n} (\sigma)\oplus R^{-}_{n+1} (\sigma), n\geq 0$.

\end{proposition}

Here, we have put $R^- _{-1} (\sigma)=R^+ _0 (\sigma)$.

\begin{proof}

This Proposition can be roughly seen from the tree of $G$, but we want to make the inclusions in the statement more precisely, using the formula \eqref{explicit form for T}.

For $(1)$, by the formula \eqref{explicit form for T}, we only need to note $\beta_K u \alpha^{-1}\in N'_{m_K}\alpha K$ for $u\in N_{n_K +1}$.

For $(2)$, as $n\geq 1$, we note firstly that $\alpha^{-n}u\alpha^{-1}\in N_{n_K} \alpha^{-(n+1)}$ for $u\in N_{n_K}$. It remains to check the following, which completes the argument of $(2)$:

$\alpha^{-n}\beta_K u \alpha^{-1}\in \begin{cases}
\alpha^{1-n}K,~~~~~~~~~~~~~~~~~~~~~~~~~~~u\in N_{n_K +2}; \\
 N_{n_K}\alpha^{-n}K,~~~~~~~~~~~~~~~~~~~~~u\in N_{n_K +1}\setminus N_{n_K +2}.
\end{cases}$

For $(3)$, let $n$ be a non-negative integer. At first, we see
\begin{center}
$\alpha^{n+1}\cdot \beta_K u\alpha^{-1}\in N'_{m_K} \alpha^{n+2}K$,
\end{center}
for $u\in N_{n_K +1} /N_{n_K +2}$. Next, we check the following, which finishes the proof of $(3)$:

\begin{center}
$\alpha^{n+1}u \alpha^{-1}\in \begin{cases}
 N'_{m_K} \alpha^n K, ~~~~~~~~~~~~~~~u\in N_{n_K +2}; \\
 N'_{m_K} \alpha^{n+1}K, ~~~~~~~~~~~u\in N_{n_K +1}\setminus N_{n_K +2}; \\
 N'_{m_K} \alpha^{n+2}K, ~~~~~~~~~~~u\in N_{n_K}\setminus N_{n_K +1}.
\end{cases}$
\end{center}
\end{proof}

\begin{remark}
The argument tells us more: for $f\in R^{+}_n (\sigma)$, we can indeed detect the parts of $T f$ which lie in $R^{+}_{n+1} (\sigma)$ and $R^{+}_{n-1} (\sigma)$.
\end{remark}

\begin{corollary}\label{formula of Tf_n}
For $n\in \mathbb{Z}\setminus \{0\}$, we have
\begin{center}
$T\cdot f_n= c_n \cdot f_n + f_{n+\delta_n}$,
\end{center}
for some constant $c_n$, and $\delta_n$ is given by
\begin{center}
$\delta_n =\begin{cases}
1 ~~~~~~~~n> 0; \\
-1~~~~~n< 0.
\end{cases}$
\end{center}
\end{corollary}

\begin{proof}

Suppose $n= -m$ is a negative integer, and we will prove that
\begin{center}
$T f_{-m}= cf_{-m}+ f_{-m-1}$
\end{center}
for some $c\in \overline{\mathbf{F}}_p$.

By (2) of Proposition \ref{image of the tree under T}, $T f_{-m}\in R^{+}_{m-1} (\sigma)\oplus R^{+}_m (\sigma)\oplus R^{+}_{m+1} (\sigma)$. As $f_{-m}$ is $I_{1, K}$-invariant, and $T$ preserves $I_{1, K}$-invariants, Lemma \ref{basis of I_1 of compact induction} and Remark \ref{basis of R-spaces} imply that
\begin{center}
$Tf_{-m}= c_{m-1}f_{-m+1} +c_m f_{-m} + c_{m+1}f_{-m-1}$,
\end{center}
for some $c_i \in \overline{\mathbf{F}}_p$.

We need to evaluate the function $Tf_{-m}$ at $\alpha^{m-1}$ and $\alpha^{m+1}$. Recall that $f_{-m}= \sum_{u\in N_{n_K}/ N_{n_K +2m}}u\alpha^{-m}\cdot f_0$, and by Proposition \ref{hecke operator formula} we get:
\begin{center}
$T f_{-m}= \sum_{u\in N_{n_K}/ N_{n_K +2m}}u\alpha^{-m}(f_{-1} + \lambda_{\beta_K, \sigma}f_1)$.
\end{center}

We need to estimate in which Cartan--Iwahori double cosets the elements $\alpha^{m-1} u\alpha^{-m}$ and $\alpha^{m+1}u \alpha^{-m}$ might belong, for $u\in N_{n_K}/ N_{n_K +2m}$.

We have firstly that:
\begin{center}
$\alpha^{m-1} u\alpha^{-m}\in\begin{cases}
 K \alpha^{-1} I_K,~~~~~~~~~~~~~~~~~~~u\in N_{2m-2+n_K}; \\
 K\alpha^{l(u)} I_K,~~~~~~~~~~~~~~~~~~u\in N_{n_K}\setminus N_{2m-2+n_K}.
\end{cases}$
\end{center}
Here, in the second inclusion above, $l(u)$ is some integer smaller than $-1$, depending on $u$. To see that, let $c(u)$ be the largest integer such that $u\in N_{c(u)}$, and the assumption $u\in N_{n_K}\setminus N_{2m-2+n_K}$ means that $d(u) :=c(u)-(2m-2)<n_K$. We now apply the equality \eqref{useful identity}:
\begin{center}
$\alpha^{m-1} u\alpha^{-m}= \beta \cdot \beta(\alpha^{m-1}u\alpha^{1-m})\cdot\alpha^{-1}= \beta\cdot u_1 \alpha^{d(u)} i_2\cdot\alpha^{-1},$
\end{center}
where $u_1\in N_{-d(u)}$, $i_2\in I_K \cap B'$. The last expression of above identity gives us that
\begin{center}
$\alpha^{m-1} u\alpha^{-m}\in K\beta\alpha^{d(u)-1}I_K$,
\end{center}
and the assertion on $l(u)$ follows.

Now we may determine the value of $c_{m-1}= Tf_{-m} (\alpha^{m-1})$. The list above immediately gives that $f_{-1} (\alpha^{m-1}u\alpha^{-m})=0$, for any $u\in N_{n_K}/ N_{n_K +2m}$. As $f_1$ is supported on $K\alpha^{-1} I_K$, the above list reduces us to look at the sum
\begin{center}
$\sum_{u\in  N_{2m-2+n_K}/N_{2m+n_K}}f_1 (\alpha^{m-1} u\alpha^{-m})= \sum_{u_1\in N_{n_K}/ N_{n_K +2}}u_1 \beta_K v_0$,
\end{center}
which is clearly zero by splitting it as a double sum, observing that $N_{n_K +1} \subseteq K^1$. In all we have shown $c_{m-1}= 0$.

\medskip
 In a similar way, we have
 \begin{center}
$\alpha^{m+1} u\alpha^{-m}\in\begin{cases}
 K \alpha I_K,~~~~~~~~~~~~~~~~~~~~~~u\in N_{2m+n_K}; \\
 K\alpha^{l'(u)} I_K,~~~~~~~~~~~~~~~u\in N_{n_K}\setminus N_{2m+n_K}.
\end{cases}$
\end{center}
Here, in the second inclusion above, $l'(u)$ is some integer smaller than $-1$, depending on $u$, which is seen by applying \eqref{useful identity} again.

Therefore, we have $f_1 (\alpha^{m+1}u \alpha^{-m})=0$, for any $u\in N_{n_K}/ N_{n_K +2m}$, and the following
\begin{center}
$f_{-1}(\alpha^{m+1}u \alpha^{-m})=\begin{cases}
v_0, ~~~~~~u\in N_{2m+n_K},\\
0, ~~~~~~~~u\in N_{n_K}\setminus N_{2m+n_K}.
\end{cases}$
\end{center}
In summary, we have proved $c_{m+1}=1$.

The exact value of $c_m$ will not be used in the paper, so we don't record it here\footnote{The exact value of $c_m$ depends on the representation $\sigma$: it is non-zero only if $\sigma$ is a character (see \cite[3.7]{X2014}).}.

\medskip
The other half of the corollary can be dealt in a completely similar manner, and we are done here.
\end{proof}

\begin{remark}
Among other things, what matters to us of the above corollary is that the coefficient of $f_{n+\delta_n}$ is $1$, especially it is non-zero.
\end{remark}

\section{Freeness of spherical Hecke modules}\label{section: proof of main}

In this section, we prove the main result of this note, as an application of Proposition \ref{hecke operator formula} and Corollary \ref{formula of Tf_n}.

Throughout this section, let $K$ be a maximal compact open subgroup of $G$, and $\sigma$ be an irreducible smooth representation of $K$.
\begin{theorem}\label{main theorem}
The maximal compact induction $\textnormal{ind}^G _K \sigma$ is free of infinite rank over $\mathcal{H}(K, \sigma)$.
\end{theorem}

\medskip
Before we continue, we give some remarks on related works in the literature.
\begin{remark}\label{GK's universal modules}
In \cite{GK14}, for a split group $\mathbf{G}$ over $F$, Elmar Gro$\ss$e-Kl$\ddot{\text{o}}$nne has proved the spherical universal module of $G= \mathbf{G}(F)$ with trivial coefficients is free, under the condition that $F$ is $\mathbf{Q}_p$, and that $G$ has connected center, and that the Coxeter number of $G$ is $p$-small (\cite[Corollary 6.1, Theorem 8.2]{GK14}).
\end{remark}

\begin{remark}
For a unramified and adjoint-type $p$-adic group (e.g., $PGL_n (F)$), Xavier Lazarus conjectured in \cite{La99} that over an algebraic closed field of characteristic different from $p$, a maximal compact induction from the trivial character is \emph{flat} over the corresponding spherical Hecke algebra, and it was proved by Bellaiche--Otwinowska in \cite{BO2003} for $PGL_3$ and $\sigma=1$ \footnote{In \cite[Remark after Corollary 6.1]{GK14}, Gro$\ss$e-Kl$\ddot{\text{o}}$nne has pointed out that freeness is indeed proved in \cite{BO2003}, even only flatness is claimed.}.
\end{remark}

We start to prove Theorem \ref{main theorem}. At first, we recall some general setting.

For $n \geq 0$, denote by $B_{n, \sigma}$ the set of functions in $\text{ind}^{G}_K \sigma$ which are supported in the ball $\mathbf{B}_n$ of the tree of radius $2n$ around the vertex $\mathbf{v}_K$ (the unique vertex on $\triangle$ stabilized by $K$). Let $C_{n, \sigma}$ be the set of functions in $\text{ind}^{G}_K \sigma$ which are supported in the circle $\mathbf{C}_n$ of radius $2n$ around the vertex $\mathbf{v}_K$. Both the set $B_{n, \sigma}$ and $C_{n, \sigma}$ are indeed $K$-stable spaces, and we may write them in term of our former notations as:
\begin{center}
$C_{0, \sigma}= R^{+}_0 (\sigma)$, $C_{n, \sigma}= R^+ _{n} (\sigma)\oplus R^- _{n-1}(\sigma)$, for $n\geq 1$.

$B_{n, \sigma}= \oplus_{k\leq n} C_{k, \sigma}$, for $n\geq 0$.
\end{center}
We prefer to define $B_{n, \sigma}$ and $C_{n, \sigma}$ in terms of the tree, as it will make some formulation of later argument easier.

\begin{lemma}\label{key lemma for freeness}
For $n\geq 0$, let $f\in B_{n+1, \sigma}$. If $T f \in B_{n+1, \sigma}$, then $f\in B_{n, \sigma}$.
\end{lemma}

\begin{proof}
Denote by $M_{n+1, \sigma}$ the subspace of $B_{n+1, \sigma}$ consisting of functions $f$ such that $Tf\in B_{n+1, \sigma}$. The assertion in the lemma means that $M_{n+1, \sigma}\subseteq B_{n, \sigma}$.

Assume there exists an $f \in M_{n+1, \sigma} \setminus B_{n, \sigma}$. As $B_{n+1, \sigma}=B_{n, \sigma}\oplus C_{n+1, \sigma}$, we may write $f$ uniquely as $f' +f''$, for some $f' \in B_{n, \sigma}$ and some $f''\in C_{n+1, \sigma}$. We see firstly that $f''= f- f' \in M_{n+1, \sigma}$, where we note that $B_{n. \sigma} \subseteq M_{n+1, \sigma}$, by the fact $B_{n, \sigma}\subseteq B_{n+1, \sigma}$ and $T B_{n, \sigma}\subseteq B_{n+1, \sigma}$ (Proposition \ref{image of the tree under T}). As $f \notin B_{n, \sigma}$, we must have $f'' \neq 0$. Thus, we have shown the space $C_{n+1, \sigma} \cap M_{n+1, \sigma} \neq 0$.

By definition, all the spaces $B_{n, \sigma}$ and $C_{n, \sigma}$ are $I_K$-stable, for $n\geq 0$. The space $M_{n+1, \sigma}$ is then also $I_K$-stable by its definition, as $T$ respects $G$-action. Therefore, we have a non-zero $I_K$-stable space $C_{n+1, \sigma} \cap M_{n+1, \sigma}$. As $I_{1, K}$ is pro-$p$, there is some non-zero function $f^\ast\in C_{n+1, \sigma} \cap M_{n+1, \sigma}$, which is fixed by $I_{1, K}$ (\cite[Lemma 1]{B-L95}).

By Remark \ref{basis of R-spaces}, the $I_{1, K}$-invariant subspace of $C_{n+1, \sigma}$ is two dimensional with the basis $\{f_{n+1}, f_{-(n+1)}\}$. By writing $f^\ast$ as a linear combination of $f_{n+1}$ and $f_{-(n+1)}$, Corollary \ref{formula of Tf_n} implies that $Tf^\ast$ does not lie in $B_{n+1, \sigma}$, a contradiction.
\end{proof}

\begin{remark}
As we will see, Lemma \ref{key lemma for freeness} is the key used to prove Theorem \ref{main theorem}. Our argument above is a bit formal up to the application of Corollary \ref{formula of Tf_n}, and as mentioned in the introduction, is essentially different from that of Barthel--Livn$\acute{\text{e}}$ (\cite[Lemma 20, 21]{B-L95}), which depends crucially on the representation $\sigma$ (of $GL_2 (k_F)$) with an explicit basis. Our strategy is applicable to give a new proof of \cite[Lemma 20]{B-L95}, and certain details will appear elsewhere. Note also that \cite[Lemma 21]{B-L95} is indeed used in the last step of the argument of \cite[Corollary 6.1]{GK14}.
\end{remark}
\medskip

\begin{proof}[Proof of Theorem \ref{main theorem}]

We proceed to complete the argument of Theorem \ref{main theorem}. Using Lemma \ref{key lemma for freeness}, by induction we find a \emph{non-empty} subset $A_n$ of $C_{n, \sigma}$, satisfying that $\sqcup_{2k+2i\leq 2n}~T^i A_k$ forms a basis of $B_{n, \sigma}$.

For $n=0$, take $A_0$ to be a basis of the space $C_{0, \sigma}$. Assume the former statement is done for $n$. Then we need to show the set $\sqcup_{\begin{subarray}{1} 2k+2i\leq 2n+2, \\ k\leq n\end{subarray}} T^i A_k$ is linearly independent.

Assume the claim is false and we have a non-trivial linear combination of elements from $\sqcup_{\begin{subarray}{1} 2k+2i\leq 2n+2, \\ k\leq n\end{subarray}} T^i A_k$. As $\sqcup_{\begin{subarray}{1} 2k+2i\leq 2n+2, \\ k\leq n\end{subarray}} T^i A_k$ is the disjoint union of $\sqcup_{\begin{subarray}{1} k+i= n+1, \\ k\leq n\end{subarray}}~T^i A_k$ and $\sqcup_{k+i\leq n}~T^i A_k$, we get a function $f$, as a linear combination of elements from $\sqcup_{k+i= n}~T^i A_k$, which lies in the ball $B_{n, \sigma}$, and satisfies that $T f\in B_{n, \sigma}$. Now Lemma \ref{key lemma for freeness} ensures that $f\in B_{n-1, \sigma}$. This means that the projection of $f$ to the circle $\mathbf{C}_n$ is zero. Note that the induction hypothesis for $n$ already implies that the set of all projections of $\sqcup_{k+i= n}~T^i A_k$ to $\mathbf{C}_n$ is a basis for $C_{n, \sigma}$, hence the former statement about $f$ forces its vanishing. We are done for the claim in the last paragraph.

We then proceed to choose a subset $A_{n+1}$ of the space $C_{n+1, \sigma}$, which completes the set $\sqcup_{\begin{subarray}{1} 2k+2i\leq 2n+2, \\ k\leq n\end{subarray}} T^i A_k$ to be a basis of $B_{n+1, \sigma}$. This is certainly possible, and we only need to complete the set of projections of $\sqcup_{\begin{subarray}{1} k+i= n+1, \\ k\leq n\end{subarray}}~T^i A_k$ to $\mathbf{C}_{n+1}$ to be a basis of $C_{n+1, \sigma}$.

We note that the subsets $A_n$ chosen above are non-empty for all $n\geq 0$: the cardinality of $A_0$ is equal to the dimension of $\sigma$. For $n\geq 1$, the cardinality of $A_n$ is equal to $\text{dim}~C_{n, \sigma}-\text{dim}~C_{n-1, \sigma}$.

\medskip
In summary, we have chosen a family of non-empty sets $A_{n}\subset C_{n, \sigma}$ so that $\cup_{n\geq 0}\sqcup_{2k+2i\leq 2n}~T^i A_k$ is a basis of the maximal compact induction $\text{ind}^{G}_K \sigma$. In particular, the infinite set $\cup_{n\geq 0} A_{n}$ is a basis of $\text{ind}^{G}_K \sigma$ over $\mathcal{H}(K, \sigma)$.
\end{proof}

The following interesting application is straightforward:

\begin{corollary}\label{inf dim}
For any non-constant polynomial $P$, the $G$-representation $\textnormal{ind}^G _K \sigma /(P(T))$ is infinite dimensional.
\end{corollary}
\begin{proof}
It suffices to consider $P$ is a linear polynomial $T-\lambda$. By writing $\textnormal{ind}^G _K \sigma$ as $\bigoplus_\mathbb{N} ~\overline{\mathbf{F}}_p [T]$, we see
\begin{center}
$\textnormal{ind}^G _K \sigma /(T-\lambda)  \cong (\bigoplus_\mathbb{N} \overline{\mathbf{F}}_p [T]) /(T-\lambda)\cong \bigoplus_\mathbb{N} \overline{\mathbf{F}}_p [T] /(T-\lambda)\cong \bigoplus_\mathbb{N} \overline{\mathbf{F}}_p$,
\end{center}
and the assertion follows.
\end{proof}

\begin{remark}
For some $P(T)$, all the irreducible quotients of $\textnormal{ind}^G _K \sigma /(P(T))$ are \emph{supersingular} representations. For such a $P(T)$, the above Corollary implies that supersingular representations containing $\sigma$ do exist.
\end{remark}

\section{Invariant norms in $p$-adic smooth principal series}\label{section: invariant norm}

In this part, we apply Theorem \ref{main theorem} to investigate the existence of $G$-invariant norms in certain locally algebraic representations of $G$ (\cite[Appendix]{ST01}). For the background and already known results on this problem, we refer the readers to the excellent survey article \cite{Sor15}. Here, we follow along the lines in \cite[section 9]{GK14}.

Let $L$ be a finite extension of $\mathbf{Q}_p$, and \emph{assume $E$ is contained in $L$}. Let $\varepsilon: B \rightarrow L^\times$ be a smooth character, and consider the principal series $P(\varepsilon)=\text{Ind}^G _B \varepsilon$ of $G$:
\begin{center}
$\{f: G\rightarrow L~\text{locally~constant}\mid f(bg)=\varepsilon(b)f(g)~\forall~ b\in B, g\in G\}$
\end{center}
and the group $G$ acts by right translation. We are interested to know \emph{whether and when} $P(\varepsilon)$ admits a $G$-invariant norm.

\medskip
Assume $\sigma_L$ is the restriction to $K$ of an irreducible $\mathbf{Q}_p$-rational representation of $G$ on a finite dimensional $L$-vector space. As $K$ is open, the representation $\sigma_L$ is still irreducible. We assume further that $\sigma_L$ is contained in $P(\varepsilon)$, i.e., the following space
\begin{center}
$\text{Hom}_{L[G]} (\text{ind}^G _K \sigma_L, \text{Ind}^G _B \varepsilon)\cong\textnormal{Hom}_{L[K]} (\sigma_L, ~\textnormal{Ind}^{G}_B\varepsilon |_K)\cong (\textnormal{Ind}^G _B \varepsilon)^{K, \sigma_L}$
\end{center}
is non-zero. By our assumption, the Satake-Hecke algebra $\mathcal{H}_L (K, \sigma_L) := \textnormal{End}_G (\textnormal{ind}^G _K \sigma_L)$ is isomorphic to $\textnormal{End}_G (\textnormal{ind}^G _K 1)$ (\cite[Lemma 2.1]{BS07}), especially it is commutative. The Iwasawa decomposition $G=BK$ implies that the above space is one-dimensional, and the natural action of $\mathcal{H}_L (K, \sigma_L)$ on it gives us a character $\chi_{\sigma_L}: \mathcal{H}_L (K, \sigma_L) \rightarrow L$.

Then, we have an induced $G$-equivariant map:
\begin{equation}\label{key map}
M_{\chi_L} (\sigma_L): =\text{ind}^G _K \sigma_L \otimes_{\mathcal{H}_L (K, \sigma_L)} \chi_{\sigma_L} \rightarrow \text{Ind}^G _B \varepsilon
\end{equation}

\medskip
Take a $K$-stable $\mathfrak{o}_L$-lattice $\sigma$ in $\sigma_L$, which is equivalent to finding a $K$-invariant norm $\mid \cdot \mid_{\sigma_L}$ on $\sigma_L$. This gives the maximal compact induction $\text{ind}^G _K \sigma_L$ a sup-norm, and then the $G$-representation $M_{\chi_L} (\sigma_L)$ inherits a canonical quotient seminorm. Note that such a seminorm may have non-zero kernel, and as a result it is \emph{not} necessarily a norm ! However, if $M_{\chi_L} (\sigma_L)$ admits a $G$-invariant norm, the previous canonical seminorm on it must be a norm (see the explanation in \cite[page 7, 8]{Sor15}).

\emph{We assume the reduction $\sigma\otimes_{\mathfrak{o}_L} k_L$ of $\sigma$ is an irreducible representation of $\Gamma_K$ over the residue field $k_L$.} We consider the restriction $\chi$ of $\chi_{\sigma_L}$ to the subalgebra $\mathcal{H}_{\mathfrak{o}_L}(K, \sigma) :=\textnormal{End}_G (\textnormal{ind}^G _K \sigma)$. Note that Theorem \ref{main theorem} implies that $\text{ind}^G _K \sigma$ is free over $\mathcal{H}_{\mathfrak{o}_L}(K, \sigma)$. \emph{When $\chi$ takes values in $\mathfrak{o}_L$}, the $\mathfrak{o}_L$-module:
\begin{center}
$M_\chi (\sigma) := \text{ind}^G _K \sigma \otimes_{\mathcal{H}_{\mathfrak{o}_L}(K, \sigma)} \chi$
\end{center}
is free (over $\mathfrak{o}_L$), which then is a $G$-stable free $\mathfrak{o}_L$-submodule of $M_\chi (\sigma)\otimes_{\mathfrak{o}_L} L\cong M_{\chi_L} (\sigma_L)$. If the map \eqref{key map} is an isomorphism, we in turn get a $G$-stable free $\mathfrak{o}_L$-lattice in the principal series $P(\varepsilon)$, as required. In summary:

\begin{theorem}\label{invariant norm}
Assume $\sigma\otimes_{\mathfrak{o}_L} k_L$ is an irreducible representation of $\Gamma_K$ over the residue field $k_L$, and

$(a)$.~ The map \eqref{key map} is an isomorphism.

$(b).~$ The character $\chi_{\sigma_L}\mid_{\mathcal{H}_{\mathfrak{o}_L}(K, \sigma)}$ takes values in $\mathfrak{o}_L$.

Then the principal series $P(\varepsilon)$ admits a $G$-invariant norm.
\end{theorem}

We end this section with some remarks on the conditions in last theorem:
\begin{remark}
1) By \cite[section 7]{Keys84}, the principal series $P(\varepsilon)$ is irreducible if $\varepsilon \neq \varepsilon^s$, and in such a case the map \eqref{key map} being an isomorphism is equivalent to require it to be injective.

2) We have a further look of the following space:
\begin{align*}
\textnormal{Hom}_{L[K]} (\sigma_L, ~\textnormal{Ind}^{G}_B\varepsilon |_K)
\cong &~\textnormal{Hom}_{L[K]} (\sigma_L, ~\textnormal{Ind}^K _{B\cap K}\varepsilon)\\
\subseteq &~\textnormal{Hom}_{\mathfrak{o}_L [K]} (\sigma, ~(\textnormal{Ind}^K _{B\cap K}\varepsilon)^{K^1})\\
\cong &~\textnormal{Hom}_{\mathfrak{o}_L [\Gamma_K]}(\sigma, ~\textnormal{Ind}^{\Gamma_K}_{\mathbb{B}}\varepsilon_0),
\end{align*}
where $\varepsilon_0 := \varepsilon \mid_{B \cap K}$. The first isomorphism is by the decomposition $G= BK$. The second inclusion is due to that the group $K^1$ acts trivially in the lattice $\sigma$. Note that the space $(\textnormal{Ind}^K _{B\cap K}\varepsilon)^{K^1}\neq 0$ implies the character $\varepsilon$ is trivial on $B\cap K^1$, and the character $\varepsilon_0$ factors through $(B\cap K) /(B\cap K^1) \cong \mathbb{B}$. Hence, we may identify $(\textnormal{ind}^K _{B\cap K}\varepsilon)^{K^1}$ with $\textnormal{Ind}^{\Gamma_K}_\mathbb{B} \varepsilon_0$, which gives the last isomorphism. The last space in the display above is non-zero if and only if (\cite{En63}):

$(1)$.~ $\sigma$ contains the trivial character of $\mathbb{U}$, i.e., $\sigma^\mathbb{U} \neq 0$,

$(2)$.~The action of $\mathbb{B}$ on $\sigma^\mathbb{U}$ contains the character $\varepsilon_0$.

By our assumption, the space $\sigma^\mathbb{U}_L$ ($=\sigma^{N_{n_K}}_L$), which contains $\sigma^\mathbb{U}$, is just the weight space underlying the highest weight of $\sigma_L$ (\cite[1', Proposition 3.4]{ST01}). Hence, $(1)$ holds automatically, and the $\mathfrak{o}_L$-module $\sigma^\mathbb{U} \neq 0$ is of rank one. Now $(2)$ is to say the group $\mathbb{B}$ acts on $\sigma^\mathbb{U}$ as the character $\varepsilon_0$.

3) \footnote{Communicated to us by Florian Herzig.}~To make the Theorem work, we have to assume that the reduction $\sigma\otimes_{\mathfrak{o}_L} k_L$ of $\sigma$ is irreducible as a representation of $\Gamma_K$. However, even if the group $K$ is hyperspecial, such a condition may fail for a general $\sigma_L$ unless its highest weight lies in the closure of the lowest alcove.
\end{remark}

\section{Appendix: Proof of Proposition \ref{basis for h_k_0}}\label{section: proof of 3.4}

\begin{proof}[Proof of Proposition \ref{basis for h_k_0}]

The argument is slightly modified from the author's thesis (\cite{X2014}). By definition, for $n\geq 1,~l\geq 0$,
\begin{center}
$\varphi_1\ast \varphi_{n}(\alpha^l)=\sum_{g\in G/K}~\varphi_1(g)\varphi_n (g^{-1}\alpha^l).$
\end{center}
As the support of $\varphi_1$ is $K \alpha K= \cup_{g\in K /(K\cap \alpha K \alpha^{-1})}~g\alpha K$, the previous sum becomes
\begin{center}
$\sum_{g\in K /(K\cap \alpha K \alpha^{-1})}~\varphi_1 (g\alpha)\varphi_n (\alpha^{-1}g^{-1}\alpha^l)$

$=\sum_{g_1\in K/ I_K}~\sum_{g_2\in N'_{m_K}/ N'_{m_K +1}}~\varphi_1 (g_1 g_2 \alpha)\varphi_n(\alpha^{-1}g_2^{-1}g_1^{-1}\alpha^l),$
 \end{center}
where, we note that $K\supseteq I_K \supseteq K\cap \alpha K \alpha^{-1}$, and that $I_K /(K\cap \alpha K \alpha^{-1})\cong N'_{m_K}/ N'_{m_K +1}$.

To proceed, we use the Bruhat decomposition $K=I_K \bigcup I_K \beta_K I_K$ to split the above sum further into two parts, say,
\begin{center}
$\sum_1=\underset{g_2\in N'_{m_K}/N'_{m_K +1}}{\sum}~\varphi_1 (\beta_K g_2 \alpha)\varphi_n(\alpha^{-1} g_2^{-1}\beta_K \alpha^l)$
\end{center}
and
\begin{center}
$\sum_2= \sum_{g_1 \in N'_{m_K -1}/N'_{m_K}}\sum_{g_2\in N'_{m_K}/N'_{m_K +1}}~\varphi_1 (g_1 g_2 \alpha)\varphi_n(\alpha^{-1}g_2^{-1}g_1^{-1}\alpha^l).$
\end{center}

Firstly we claim $\sum_1$ is always $0$. As $\varphi_1\in \mathcal{H}_K(\sigma)$,  $\sum_1$ is simplified as
\begin{center}
$\sum_1=\underset{g_2 \in N'_{m_K}/N'_{m_K +1}}{\sum}~\sigma(\beta_K)j_\sigma \varphi_n(\alpha^{-1} g_2^{-1}\beta_K \alpha^l).$
\end{center}

We note that $\alpha^{-1} g_2^{-1}\beta_K\alpha^l\in K \alpha^{-(l+1)}K$, hence we only need to consider the case that $l+1=n$. In this case, the sum $\sum_1$ is reduced to

\begin{center}
$\sum_1= \sum_{g_2} \sigma (\beta_K) j_\sigma \sigma (\beta_K) j_\sigma,$
\end{center}
which is equal to $0$, as it sums over the same term $q^{c_K}$ times.

For the remaining $\sum_2$, we note the part $\sum'_2$ for which $g_1\in N'_{m_K -1}\setminus N'_{m_K}$ is equal to $0$. A simple calculation using  \eqref{useful identity} gives
\begin{center}
$\alpha^{-1}g^{-1}_2 g^{-1}_1 \alpha^l= g'\alpha^{-(l+1)}g'',$
\end{center}
for some $g'\in N_{n_K +2}$ and $g''\in K$. As a result, we have $\sum'_2=0$, if $l\neq n-1$. When $l=n-1$, one can rewrite $\sum'_2$ as
\begin{center}
$\sum'_2= \sum_{g_2}~(\sum_{g_1}f'),$
\end{center}
in which $f'$ \footnote{For the concrete form of $f'$, one needs to distinguish $l=0$ or not.} is a function only related to $g_1$. We get $\sum'_2=q^{c_K}\cdot (\sum_{g_1}f') =0$.

The other part of $\sum_2$, denoted by $\sum''_2$, depends on the values of $l$ and $n$:

\begin{center}
$\sum_{g_2\in N'_{m_K}/N'_{m_K +1}}~\varphi_1 (g_2 \alpha)\varphi_n(\alpha^{-1}g_2^{-1}\alpha^l)$
$=\begin{cases}
j_{\sigma}, ~~~~~~~l=n+1,\\
c \cdot j_{\sigma}, ~~l=n,\\
0,~~~~~~~~\text{otherwise},
\end{cases}$
\end{center}
where $c$ is a constant.

\medskip
From the definition of $\varphi_1$, the sum $\sum''_2$ is reduced to
\begin{center}
$\sum''_2= \sum_{g_2} j_\sigma \varphi_n (\alpha^{-1}g^{-1}_2\alpha^l).$
\end{center}
If $l=0$, a term in $\sum''_2$ is non-zero only if $n=1$; but in this case, the sum itself is clearly zero.

We assume $l\geq 1$. For the term in the sum $\sum''_2$ corresponding to $g_2\in N'_{m_K +1}$, it becomes $j_\sigma \varphi_n (\alpha^{l-1})$, as $\alpha^{-1}g^{-1}_2 \alpha\in I'_{1, K}$ and note that $j_\sigma$ is trivial on $\sigma(I'_{1, K})$. Hence, such term is non-zero only if $l=n+1$, and in this case it is equal to $j^{2}_\sigma=j_\sigma$. For the remaining terms in $\sum''_2$, we write $g^{-1}_2$ as $g$ for short, and put $g^\ast=\beta_K g \beta_K$. By using \eqref{useful identity} again we get
\begin{center}
$\alpha^{-1}g \alpha^l= \alpha^{-1}\beta_K g^\ast \beta_K \alpha^l\in K^1 h\beta_K \alpha^l K^1$
\end{center}
for some diagonal matrix $h\in K$ depending on $g=g^{-1}_2$. Therefore, each term is non-zero only if $l=n$, and in that case we get the sum $\sum''_2$ as
\begin{center}
$\sum_{h}~j_{\sigma}\sigma(h\beta_K) j_{\sigma}.$
\end{center}

As we have just shown that the other parts of the whole sum contribute zero to its value at $\alpha^n$, the above sum must be a multiple of $j_\sigma$, say $c\cdot j_\sigma$ for some constant $c$. In all, we are done.
\end{proof}

\section*{Acknowledgements}
 Part of this note grows out of the author's PhD thesis (\cite{X2014}) at University of East Anglia, and he would like to thank Professor Shaun Stevens for helpful comments on a first draft. The author was supported by a Postdoc grant from Leverhulme Trust RPG-2014-106.

\bibliographystyle{amsalpha}
\bibliography{new}

\texttt{Warwick Mathematics Institute, University of Warwick, Coventry, CV4 7AL, UK}

\emph{E-mail address}: \texttt{Peng.Xu@warwick.ac.uk}

\end{document}